\documentclass[english,a4paper]{article}
\usepackage[T1]{fontenc} 
\usepackage{textcomp,lmodern} 
\usepackage[utf8]{inputenc} 
\usepackage[babel=true]{microtype} 
\usepackage{babel} 
\usepackage{mathtools} 
\usepackage{amssymb} 
\usepackage{amsthm} 
\newtheorem{lemma}{Lemma}[section]

\newtheorem{corollary}[lemma]{Corollary}
\newtheorem{proposition}[lemma]{Proposition}
\newtheorem{theorem}[lemma]{Theorem}
\theoremstyle{definition}
\newtheorem{definition}[lemma]{Definition}

\newtheorem{remark}[lemma]{Remark}
\newtheorem{example}[lemma]{Example}

\DeclareMathOperator\Rauz{Rauz} 
\DeclareMathOperator{\prefix}{prefix}
\DeclarePairedDelimiter\abs{\lvert}{\rvert} 
\newcommand*\A{{\mathcal A}} 
\newcommand*\Lang{{\mathcal L}} 
\newcommand*{\field}[1]{\mathbf{#1}} 
\newcommand*\orR{{\vec\Rauz}} 
\newcommand*{\setst}[2]{\{#1\ |\ #2\}}
\newcommand*\Z{\field{Z}} 
\newcommand*\N{\field{N}} 
\newcommand*\R{\field{R}} 
\newcommand*\Q{\field{Q}} 
\newcommand*\G{\mathcal{G}} 
\usepackage{hyperref}
\newcommand*{\titre}{Limits of Rauzy graphs of languages with subexponential complexity}
\title{\titre}
\author{Paul-Henry Leemann, Tatiana Nagnibeda,\\ Alexandra Skripchenko, Georgii Veprev}
\date{\today}
\hypersetup{pdftitle={\titre}, pdfauthor={Paul-Henry Leemann, Tatiana Nagnibeda, Alexandra Skripchenko, Georgii Veprev}, pdfsubject={Symbolic dynamic, Rauzy digraphs, Benjamini-Schram limit}} 
\begin{document}
\maketitle
\begin{abstract}
To a subshift over a finite alphabet, one can naturally associate an infinite family of finite graphs, called its Rauzy graphs. 
We show that for a subshift of unbounded subexponential complexity the Rauzy graphs converge in the sense of Benjamini-Schramm if and only if their empirical spectral measures converge, which is equivalent to the complexity function $p(n)$ to satisfy $\lim_n\frac{p(n+1)}{p(n)} = 1$. The limit in this case is the infinite line $\mathbf Z$.
If the subshift is moreover uniquely ergodic, the limit of labelled Rauzy digraphs if it exists can be identified with the unique invariant measure. We also consider a non uniquely ergodic example recently studied by Cassaigne and Kabor\'e and identify a continuum of invariant measures with subsequential limits of labelled Rauzy digraphs. 
\end{abstract}
\section{Introduction}

One of the central questions in the study of a dynamical system is that of describing its invariant measures. Many dynamical systems admit a natural symbolic representation as subshifts, i.e., acting as a shift on an invariant subspace in the space of infinite words over a finite alphabet. A popular way to study their complexity is by associating to the subshift an infinite family of finite graphs that describe the local structure of its orbits; they were introduced by G.~Rauzy in \cite{R} and are called the \emph{Rauzy graphs}. The purpose of our work is to connect the Rauzy graphs associated with a dynamical system with its invariant measures and to provide a new, graphical realization of the latter by considering the Benjamini-Schramm limits of the Rauzy graphs.  

Given an arbitrary sequence of finite graphs of uniformly bounded degree,  
I.~Benjamini and O.~Schramm~\cite{BeSch} equipped it with a natural distributional limit 
which has since become a central object of investigation in many different contexts under the name of Benjamini-Schramm limit. The limiting object is not a graph, but rather a probability measure on the space of (isomorphism classes of) rooted graphs obtained by choosing uniformly at random a root in each of the graphs in the sequence and by considering all possible limits in the space of rooted graphs with local topology. The key fact is that it is unimodular; see Section \ref{Section:Rauzy} for a precise definition. In this paper, when talking about convergence of graphs, we will always mean convergence in the sense of Benjamini-Schramm. In the special case when the limiting measure of a sequence of finite graphs $(G_i)_i$ is supported on one (unimodular, hence transitive) graph $G$ we will say that $(G_i)_i$ converge to $G$.

Applying this construction to a sequence of Rauzy graphs associated with a subshift, it is natural to ask how their Benjamini-Schramm limit is related to the original dynamical system and its invariant measures. In this paper we provide some answers to this question in the case of subshifts of subexponential complexity. It turns out that in this case convergence of graphs can be detected by spectral convergence.
In various examples, the Benjamini--Schramm limits of Rauzy graphs give us an alternative description of the invariant measures for the original dynamical systems (see Section~\ref{Sec:Applications}). 

Let $\Lang$ be a language over a finite alphabet $\mathcal A$ and denote by $p(n)$ its \emph{complexity function}; it counts the number of words of length $n$ that belong to $\Lang$. Let us recall that $\Lang$ is \emph{factorial} if it is closed under taking subwords and \emph{prolongable} if for every word $w\in\Lang$ there exists letters $a$ and $b$ in the alphabet $\mathcal A$ such that $awb$ is still in $\Lang$. A language is \emph{almost prolongable} if the proportion of prolongable words among all words of length $n$ tends to $1$ as $n$ tends to infinity.

For a factorial language $\Lang$, we denote by~$\orR(n)$ the $n$\textsuperscript{th} \emph{Rauzy digraph}  of~$\Lang$.
That is, the vertices of $\orR(n)$ are the words of length~$n$, while the arcs are words of length~$n+1$. More precisely, for $w$ of length $n-1$, the word $awb$ is the arc from $aw$ to $wb$. Given two words $u$ and $v$ connected by an arc, the arc can be labelled by the last letter of $v$. We can thus speak about \emph{labelled Rauzy digraphs} or, if we forget about the labelling of the arcs, about \emph{unlabelled Rauzy digraphs}.
It is also possible to forget about orientation, replacing every arc by an undirected edge, and obtain labelled and unlabelled \emph{Rauzy graphs} $\Rauz(n)$.

Our interest in the Benjamini-Schramm limits of families of Rauzy graphs initiated in the observation in \cite{GrLN} that the limit of De Brujin graphs $\{\mathcal B(k,N)\}_N$ (i.e., the Rauzy graphs of the full shift over the alphabet $\{0,...,k-1\}$) is the (delta-measure on the) Diestel-Leader graph $DL_{k,k}$ which is a Cayley graph of the lamplighter group $\mathbf Z_k\wr \mathbf Z$. Generalizing this result, the first-named author showed~\cite{Leemann} that for languages coming from subshifts of finite type, the Benjamini-Schramm limit of Rauzy graphs is a measure supported on horocyclic products of trees that can be constructed from the subshift. On the other hand, Drummond studied the subshifts of linear complexity ~\cite{Drummond} and proved that the Benjamini-Schramm limit of their unlabelled Rauzy digraphs is the (delta-measure on the) bi-infinite directed path $\vec \Z$ (see also~\cite{Frid, Se18, GLN} for related statements about Rauzy graphs of subshifts of linear complexity).

The question about the Benjamini-Schramm limit of de Brujin graphs was motivated by spectral computations that appeared in \cite{BDh}. Indeed, if a sequence of finite graphs converges, then the associated empirical spectral measures (aka counting measures on the spectra of the corresponding adjacency matrices)  converge to the expected spectral measure on their Benjamini-Schramm limit, see Section  \ref{Section:Rauzy} for details. Spectral convergence however does not in general imply graph convergence.

The main result of this paper is a characterisation of convergence, spectral and in the Benjamini-Schramm sense, of Rauzy graphs of languages with subexponential complexity.

\begin{theorem}\label{Thm:Main}
Let $\Lang$ be a factorial and almost prolongable language with unbounded complexity function $p(n)$. The following are equivalent:
\begin{enumerate}
\item $p(n)$ grows subexponentially and the limit of the Rauzy graphs $\Rauz(n)$ exists.
\item The Rauzy graphs $\Rauz(n)$ converge to the line $\Z$.
\item $p(n)$ grows subexponentially and the empirical spectral measures $\mu_{\Rauz(n)}$ converge.
\item The empirical spectral measures $\mu_{\Rauz(n)}$ converge to the measure that has density $\frac{1}{\pi\sqrt{4-x^2}}dx$ with respect to the Lebesgue measure on~$[-2, 2]$.
\item $p(n)$ satisfies $\lim_n\frac{p(n+1)}{p(n)}=1$.
\item $p(n)$ grows subexponentially and the limit of the Rauzy digraphs $\orR(n)$ exists.
\item The Rauzy digraphs $\orR(n)$ converge to the directed line $\vec \Z$.
\end{enumerate}
\end{theorem}
Theorem~\ref{Thm:Main} directly applies to languages coming from bi-sided subshifts, as they are both factorial and prolongable, hence, almost prolongable.
While the language of a one-sided subshift is in general not almost prolongable, it is the case for a broad class of subshifts. Surjective subshifts, in particular minimal ones, have prolongable language of subwords. Moreover, if a subshift has a dense orbit or, more generally, if it has a finite set with dense orbit then the language is almost prolongable (see Subsection~\ref{Subsection:OneSided}). 

 Using, in particular, \cite{MR1466182}, we deduce the following sufficient condition for the convergence of Rauzy graphs of subshifts. Here and in what follows, for two non-negative sequences~$f$ and~$g$, we write $f = O(g)$ if $\limsup_n f(n)/g(n) < \infty$, and $f = \Theta(g)$ if both relations~$f = O(g)$ and~$g = O(f)$ are satisfied. We will also write~$f \sim g$ if $\lim_n f(n)/g(n) = 1$.
\begin{corollary}\label{Cor:SubshiftCor}
Let $\Sigma$ be a bi-sided subshift or a one-sided subshift.
For one-sided subshifts, suppose moreover that $\Sigma$ has a finite subset with dense orbit.
Suppose that $\Sigma$ has unbounded complexity and that at least one of the following three conditions is satisfied:
\begin{enumerate}
\item $p(n)=O(n^\alpha)$ for some $\alpha<\frac43$,
\item $p(n)=\Theta(n^\beta)$ for some $\beta<\frac32$,
\item $p(n)\sim n^\gamma$ for some real number $\gamma$.
\end{enumerate}
Then the Rauzy digraphs of $\Sigma$ converge to $\vec \Z$ and its Rauzy graphs converge to $\Z$.
\end{corollary}

In Section~\ref{Sec:Applications}, we will see that Theorem~\ref{Thm:Main} and Corollary~\ref{Cor:SubshiftCor} apply to many interesting dynamical systems of different origins, including interval exchange transformations and interval exchange transformations with flips, interval translation maps, polygonal and polyhedral billiards, Toeplitz subshifts, multidimensional continuous fractions and more.

Since Rauzy digraphs come with a natural labelling, we are also interested in the convergence of the labelled Rauzy digraphs. If the limit exists, then Theorem~\ref{Thm:Main} implies that it will be a probability measure on rooted labelled digraphs supported on~$\vec\Z$. It can be thought of as a probability measure on the set of bi-infinite words in the alphabet $\mathcal A$ invariant under the shift (as a consequence of unimodularity, see Section \ref{Section:Rauzy} for details).
It is therefore not surprising that unique ergodicity is a natural condition sufficient for the existence of the Benjamini-Schramm limit of labelled Rauzy digraphs. 

In the context of unique ergodicity, it is natural to consider identifying the Benjamini-Schramm limit of labelled Rauzy digraphs for topologically transitive subshifts.
Any such subshift (one-sided or two-sided) can be represented as follows. For an infinite word $\omega$, let us denote by~$\Sigma_\omega$  the closure of the orbit of~$\omega$ under the shift map~$S$. That is, $\Sigma_\omega$ is the smallest one-sided subshift (if $\omega$ is a one-sided infinite word) or the smallest bi-sided subshift (if $\omega$ is a bi-sided infinite word) containing $\omega$. 
It is known that the subshift $(\Sigma_\omega,S)$ is uniquely ergodic if and only if the word $\omega$ has \emph{uniform frequencies} (see Section~\ref{Sec:Labelled}).
The natural extension $(\tilde\Sigma_\omega, S)$ of $(\Sigma_\omega,S)$ is a bi-sided subshift that consists of all infinite words~$\tau$ such that every finite subword of~$\tau$ is also a subword of~$\omega$. If the initial subshift is already bi-sided, then~$(\tilde\Sigma_\omega, S)$ and~$(\Sigma_\omega,S)$ coincide. It~is known that the uniform frequences property, hence, unique ergodicity is preserved under passing to the natural extension. 
In this context we have
\begin{theorem}\label{Thm:MainUnifFreq}
Let~$\omega$ be an infinite word over a finite alphabet such that the subshift $(\Sigma_\omega,S)$ is uniquely ergodic. Assume that its complexity function is unbounded and satisfies $\lim_{n}\frac{p(n+1)}{p(n)} = 1$.
Then the Benjamini--Schramm limit of the labelled Rauzy digraphs exists and can be identified with the unique $S$-invariant measure on $(\tilde \Sigma_\omega, S)$.
\end{theorem}
 In the case of subshifts of linear complexity, this was proved in \cite{Drummond}.
\begin{remark}
    The uniform frequencies property can be defined not only for transitive subshifts but for arbitrary languages over a finite alphabet as well. Following the same arguments, it is possible to obtain a similar result for almost prolongable languages with uniform frequencies.
\end{remark}

In the case of a non-uniquely ergodic subshift, one can consider subsequential limits; we show that we can always obtain an invariant measure for a subshift as such a limit with respect to some subsequence of labelled Rauzy digraphs (see Proposition~\ref{Prop:SubMeasure} in Section~\ref{Sec:Labelled}). In Section~\ref{counterexample} we present a  concrete example of a non-uniquely ergodic system where the limit of labelled Rauzy digraphs does not exist and where subsequential limits realize a continuum of distinct invariant measures.
\paragraph{Plan of the article}
Section~\ref{Section:Rauzy} contains all the necessary definitions and the proof of Theorem~\ref{Thm:Main}. 
It also contains the proof of Corollary~\ref{Cor:SubshiftCor}, which is split between Corollaries~\ref{Cor:Asympt} and~\ref{Cor:SmallCompl} for bi-sided subshifts.
The statements about one-sided subshifts to which we can apply Theorem~\ref{Thm:Main} and Corollary~\ref{Cor:SubshiftCor} are proved with the help of an additional technical result: Proposition~\ref{Cor:DO}.
Section~\ref{Sec:Labelled} deals with the limits of labelled Rauzy graphs.
Section~\ref{Sec:Applications} is devoted to applications of the above to various concrete examples of dynamical systems.
Finally, Section~\ref{counterexample} contains an example of a non-uniquely ergodic subshift where we identify subsequential limits of labelled Rauzy graphs and show that they give us a continuum of invariant measures. 

\section*{Acknowledgements} This work was supported by the Research Preparation Grant -- Leading House Geneva (2022). 
A. S. heartily thanks University of Geneva for its hospitality and A. Frid for several interesting discussions. 
%
%
%
%
%
\section{Limits of Rauzy graphs of languages with subexponential complexity}\label{Section:Rauzy}
In the following, $k$ will be an integer greater than $1$ and $\A$ will be a set of cardinality $k$, which we will call an \emph{alphabet}.
Elements of $\A$ are called \emph{letters}, while elements of the free monoid $\A^*$ (where the operation is the concatenation) are called \emph{words}.
A \emph{language} $\Lang$ is simply a subset of $\A^{*}$. We will denote by  $\Lang_n = \A^n\cap\Lang$ the set of words of length $n$ in $\Lang$ and by $p(n)=p_\Lang(n)$ the \emph{complexity} function of $\Lang$, that is, for every $n$, $p(n)=\abs{\Lang_n}$ is the number of words of $\Lang$ of length $n$. 

A language $\Lang$ is \emph{factorial} if for any word $u$ in $\Lang$, any subword $v$ of $u$ is still in~$\Lang$.
A word $v$ of $\Lang$ is \emph{left-prolongable}, respectively \emph{right-prolongable}, if there exists $a\in\A$, respectively $b\in\A$, such that $av$, respectively $vb$ is still in~$\Lang$.
\begin{definition}\label{Def:Prolongable}
A language is \emph{prolongable} if any word $v\in\Lang$ is both left-prolongable and right-prolongable.
We will say that $\Lang$ is \emph{almost prolongable} if both the proportion of left prolongable words of length $n$ and the proportion of right prolongable words of length $n$ tend to $1$ when $n$ goes to infinity.
\end{definition}
\begin{remark}
In the following, for a factorial language $\Lang$, we will always assume that $p(1)=k$, as otherwise this implies that a letter of $\A$ is not used in $\Lang$ and that $\Lang$ is hence a language on a smaller alphabet $\A'\subset\A$.
\end{remark}
As mentioned in the introduction, we denote by $\orR(n)=\orR_\Lang(n)$ the $n$\textsuperscript{th} \emph{Rauzy digraph} of~$\Lang$.
That is, the vertex set of $\orR(n)$ is $\Lang_n$ the set of words of length $n$, and there is an arc between~$v$ and~$w$ if there exist letters $a,b\in\A$ and $u\in\Lang_{n-1}$ such that $v=au$, $w=ub$, and $aub$ is in~$\Lang_{n+1}$.
It naturally gives rise to the \emph{Rauzy graph} $\Rauz(n)$ where every arc $(au, ub)$ is replaced by an edge~$\{au,ub\}$. We will also consider the \emph{labelled Rauzy (di)graph} obtained by putting the label~$b$ on the arc from $au$ to $ub$.

Note that the number of vertices of the (labelled) Rauzy (di)graph is exactly~$p(n)$.
If $\Lang$ is factorial, then the arc set of the Rauzy digraph $\orR_\Lang(n)$ can be identified with $\Lang(n+1)$ and, moreover, $\orR_\Lang(n+1)$ is a spanning subgraph of the line graph of $\orR_\Lang(n)$. 

A \emph{rooted graph} is a connected graph with a specified vertex, which is called the root.

Let us denote by $\G$ the space of (isomorphism classes of) rooted graphs of uniformly bounded degree, where two rooted graphs are close if they coincide on a big ball around their roots.
One defines similarly the spaces of rooted digraphs, of rooted labelled graphs and of rooted labelled digraphs.
We will write $\G$ as a placeholder for any of the four above spaces.
The topology of $\G$ is induced by the metric: $d((G_1,v_1),(G_2,v_2))=2^{-r}$ where $r$ is the biggest radius such that $G_1$ and $G_2$ have isomorphic balls of radius $r$ around their respective root. The metric is complete and separable. Moreover, the space $\G$ of rooted (di)graphs (labelled or not) of uniformly bounded degree is compact.
Since here we study limits of Rauzy graphs of  languages over an alphabet with $k$ letters, all our graphs are of degree bounded by $k$.
Since $\G$ is separable and complete metrizable, the space $\mathbf P(\G)$ of Borel probability measures on $\G$ is compact in the weak topology.

For a sequence $(G_i)_i$ of finite graphs in $\G$ let us choose the root in each of them uniformly at random. This means that we define, for each $i$, the measure  
\begin{equation}
	\mu_i \coloneqq \frac{1}{\abs{G_i}}\sum_{v\in G_i}\delta_{[(G_i,v)]}
\end{equation} 
in the space $\mathbf P(\G)$. If the limit of the sequence $\mu_i$ in $\mathbf P(\G)$ exists then it is called the  \emph{Benjamini-Schramm limit of the sequence} $(G_i)_i$ (see~\cite{BeSch}). 

An important property of the Benjamini-Schramm limit is its \emph{unimodularity} that we will now briefly define (see~\cite{AlLy} for details).

\begin{definition}\label{unimod}
A \emph{flagged rooted graph} is a rooted graph $(G, v)$ with a chosen edge from~$E(G)$ incident to the root and oriented such that~$v$ is its tail. Consider the space of flagged rooted graphs of uniformly bounded degree $\mathcal{FG}$. There exists a natural  map $\pi : \mathcal{FG} \to \mathcal{G}$ that simply forgets the flag. 
Given a random rooted graph $\mu \in \field{P}(\mathcal{G})$, we consider a measure~$\mu^f$ on~$\mathcal{FG}$ given by assigning to each flagged rooted graph the mass of the underlying rooted graph. 
Consider the operator~$R$ on the space $\mathcal{FG}$ acting by replacing the root to the head of the flag, and reversing the orientation of the flag. We say that random rooted graph~$\mu$ is \emph{unimodular} if the measure $\mu^f$ is invariant under~$R$.
The same definition applies to digraphs and to labelled graphs and labelled digraphs. There one has to choose a flag among all the edges incident to the root (independently of their orientation and labels). 
\end{definition}

 A specific case of Benjamini-Schramm convergence is when the limit is a measure $\mu$ supported on one rooted (di)graph $(G,v)$ with $G$ unimodular and hence vertex-transitive. As Theorem \ref{Thm:Main} shows, this is the case for Rauzy graphs of languages of subexponential complexity, as their limit, if exists, is (the delta-measure on) the infinite line $\mathbf{Z}$.   In this particular case the limit of labelled Rauzy digraphs becomes a measure on the set $\mathcal A^{\mathbf{Z}}$ of arc-labellings of the oriented line with letters from $\mathcal{A}$, and we have the following Lemma.

 \begin{lemma}\label{inv}  Let~$\mu$ be a unimodular measure on the space of labelled digraphs supported on $A^\Z$ that we realize as the set of arc-labellings of the oriented infinite line $\vec{\mathbf{Z}}$. Then~$\mu$ is shift invariant. 
\end{lemma}
\begin{proof}
    The measure $\mu^f$ is supported on two copies of~$A^\Z$, that is, on the set $A^\Z \times \{+1, -1\}$, where~$+1$~corresponds to the flag pointing to the right and~$-1$ corresponds to the flag pointing to the left. Moreover, the measure $\mu^f$ is exactly the product~$\mu \otimes \{\delta_{+1}, \delta_{-1}\}$. Let~$T$ be the shift on $A^\Z$. In terms of~$T$, the operator~$R$ can be expressed as follows
    $$
    R(x, \sigma) = (T^\sigma x, -\sigma), \quad x \in A^\Z, \ \sigma \in \{{+1}, {-1}\}.
    $$
    Let $B$ be a measurable subset of $A^\Z$. Then due to $R$-invariance of $\mu^f$ we have 
    $$
    \mu(T(B)) = \mu^f (T(B)\times \{-1\}) = \mu^f (B \times \{+1\}) = \mu(B).
    $$
    Therefore, the measure~$\mu$ is invariant under~$T$.
\end{proof}
 
Given a finite graph~$G$, let~$A_G$ be its adjacency matrix and $\sigma(A_G) \subset \R$ be the set of eigenvalues of~$A_G$. Let~$m_\lambda$ be the multiplicity of an eigenvalue $\lambda \in \sigma(A_G)$. \emph{The empirical spectral measure}~$\mu_G$ of~$G$ is the following measure 
\[
\mu_G = \frac{1}{\abs{V(G)}} \sum_{\lambda \in \sigma(A_G)} m_\lambda \delta_\lambda,
\]
where $\delta_\lambda$ is the Dirac measure at point~$\lambda$.   It can be shown that for a sequence of finite graphs  $(G_i)_i$ converging in the Benjamini-Schramm sense to a random rooted graph~$\nu \in \mathbf P(\G)$, the empirical spectral measures $\mu_{G_i}$ converge to the \emph{expected spectral measure} of $\nu$ which is the average of spectral measures $\mu_{G,v}$ of a graph $G$  computed at the vertex $v$, where the rooted graph $(G,v)$ belongs to the support of $\nu$; see, e.g., \cite{MoWo} for spectral measures on infinite graphs and \cite{ATV} for expected spectral measures on random rooted graphs. (Note that expected spectral measure can be defined for any random rooted graphs and for a finite graph with the uniform root, the expected spectral measure and the empirical spectral measure coincide.) In general, convergence of spectral measures $\mu_{G_i}$, for a sequence of finite graphs $(G_i)_i$, does not imply convergence of  $(G_i)_i$. 

As follows from \ref{Thm:Main}, the only random rooted graph that we have to deal with in this paper is the delta-measure on a transitive graph, namely, the infinite line $\mathbf{Z}$, in which case there is no need to take the average and the limit is just the spectral measure of the adjacency operator on $\mathbf{Z}$ which is independent on the vertex $v\in \mathbf Z$ and which can be easily found, see, e.g., \cite{Ha}[Prop.3.2].
%
%
%
%
%
\subsection{Rauzy (di)graphs converging to the (directed) line}
The aim of this subsection is to characterize languages for which the sequence of Rauzy (di)graphs converges to the (directed) line. We will prove Theorem~\ref{Thm:Main}. 

For a language $\Lang$,  let the functions $e_l(n)$ and $e_r(n)$ count the number of words of length $n$ that are not left, respectively not right, prolongable while $s_l(n)$ and $s_r(n)$ will count words that are left, respectively right, prolongable but not in a unique way.
Finally, let $r_l(n)$ and $r_r(n)$ count the number of words that are left, respectively right, prolongable in a unique way, while $r(n)=\abs{\setst{v\in\A^n\cap\Lang}{\abs{\A v\A\cap\Lang}=1}}$ counts the number of words that are both left and right-prolongable in a unique way.
With these notations, $\Lang$ is prolongable if and only if $e_r(n)\equiv e_l(n)\equiv 0$ and almost prolongable if and only if $\lim_n\frac{e_l(n)}{p(n)}=\lim_n\frac{e_r(n)}{p(n)}=0$.

First, we show an upper bound on the number of connected components of $\orR(n)$ that are directed cycles.
\begin{lemma}\label{Lemma:Cycles2}
Let $\Lang$ be a language on a finite alphabet of cardinality $k$.
Denote by $c_r(n)$ the number of (weakly) connected components of $\orR(n)$ that are directed cycles of length $r$.
Then for $n\geq r$ we have
\[
	c_r(n)\leq k^r.
\]
\end{lemma}

\begin{proof}
Let $n$ be an integer greater than $r$ and let $C$ be a connected component of $\orR(n)$ which is a directed cycle.
Choose an arbitrary vertex $v=x_1\cdots x_n$ in~$C$.
Then the vertex following $v$ in the cycle is of the form $x_2\cdots x_nx_{n+1}$ and when closing the cycle we obtain $x_1\cdots x_n=x_{1+r}\cdots x_nx_{n+1}\cdots x_{n+r}$ since $n\geq r$.
In particular, for every $1\leq i\leq n$ we have the relation $x_i=x_{i+r}$ and the cycle $C$ is uniquely determined by the sequence $x_1\cdots x_r$, which gives us the desired result.
\end{proof}

Let us recall that the function $r(n)$ counts the number of words of length $n$ that are both left and right-prolongable in a unique way.
\begin{lemma}\label{Lemma:Technical}
Let $\Lang$ be any language over a finite alphabet.
Then the following are equivalent:
\begin{enumerate}
\item The Rauzy digraphs $\orR(n)$ converge to the directed line $\vec \Z$,\label{Cond:1}
\item The Rauzy graphs $\Rauz(n)$ converge to the line~$\Z$  and $\Lang$ is almost prolongable,\label{Cond:2}
\item $p(n)$ is unbounded and $\lim_n \frac{r(n)}{p(n)}=1$.\label{Cond:3}
\end{enumerate}
\end{lemma}
\begin{proof}
\eqref{Cond:2}$\implies$\eqref{Cond:3}.
Suppose that $\Rauz(n)$ converge to $\Z$. This is equivalent to the fact that for every $r>1$ and every $\epsilon>0$ there exists $N=N(r,\epsilon)$ such that for $n\geq N$, for at least $1-\epsilon$ of the vertices of $\Rauz(n)$ the ball of radius $r$ around $v$ is a path of length $2r$.
On the one hand, for $n\geq N(r,\frac14)$ the graph $\Rauz(n)$ has at least $1$ vertex with the ball as described above, and hence at least $2r$ vertices, which implies the first condition.
On the other hand, for $n\geq N(2,\epsilon)$, we have $\frac{t(n)}{p(n)}\geq 1-\epsilon$, where $t(n)$ is the number of vertices of $\Rauz(n)$ that have degree $2$.
This show that $\lim_n \frac{t(n)}{p(n)}=1$.
Finally, $t(n)=r(n)+q_i(n)+q_o(n)$, where $q_i(n)$ counts the number of vertices of $\orR(n)$ of in-degree $2$ and out-degree $0$, while $q_o(n)$ counts the number of vertices of out-degree $2$ and in-degree $0$.
Since $\Lang$ is almost prolongable, both $\frac{q_i(n)}{p(n)}$ and $\frac{q_o(n)}{p(n)}$ go to $0$ when $n$ goes to infinity, which finishes the proof.

\eqref{Cond:1}$\implies$\eqref{Cond:3}.
If the digraphs $\orR(n)$ converge to $\vec \Z$, then the same proof as above shows that $p(n)$ is unbounded and that $\frac{r(n)}{p(n)}\geq 1-\epsilon$ for $n\geq N(2,\epsilon)$, without supposing that $\Lang$ is almost prolongable.

\eqref{Cond:3}$\implies$\eqref{Cond:1}.
We now show that the conditions on $p(n)$ and $r(n)$ imply the convergence of the Rauzy digraphs.
Let $r$ be an integer.
For $n\geq r$, the number of non-biregular vertices of $\orR(n)$ is $p(n)-r(n)$.
Since every vertex has at most $2k$-neighbours, the number of vertices of $\orR(n)$ which are at distance at most $r$ of a non-biregular vertex is bounded above by $(2k)^{r+1}\bigl(p(n)-r(n)\bigr)$.
In consequence, for at least
\[
	1-\frac{(2k)^{r+1}\bigl(p(n)-r(n)\bigr)}{p(n)}
\]
of the vertices of $\orR(n)$, the ball of radius $r$ around $v$ consists only of biregular vertices.
The only connected rooted digraphs of radius at most $2r$ with only biregular vertices --- excepts maybe for vertices at distance $r$ from the root which might have in or out-degree $0$ --- are the directed paths of length $2r$ and the directed cycles of length less than $2r$.
By Lemma~\ref{Lemma:Cycles2} there are at most
\[
\sum_{i=1}^{2r}i\cdot k^{i}\eqqcolon M
\]
vertices of $\orR(n)$ contained in a connected component of $\orR(n)$ isomorphic to a cycle of length at most $2r$.
Altogether, the proportion of vertices of $\orR(n)$ such that the balls of radius~$r$ around them is a directed path is at least
\[
	1-\frac{(2k)^{r+1}\bigl(p(n)-r(n)\bigr)}{p(n)}-\frac{M}{p(n)}.
\]
By assumptions on $p(n)$ and $r(n)$ this quantity tends to $1$ when $n$ goes to infinity, which proves the Benjamini--Schramm convergence of the $\orR(n)$ to $\vec \Z$.

\eqref{Cond:1}$\implies$\eqref{Cond:2}.
It is clear that the convergence of the Rauzy digraphs implies the convergence of the Rauzy graphs. It hence remains to show that it also implies that $\Lang$ is almost prolongable.
It follows from the definition that for $*\in\{l,r\}$ we have
\begin{equation}\label{Eq:I1}
	r(n)\leq r_*(n)= p(n)-s_*(n)-e_*(n).
\end{equation}
But then, using that the convergence of $\orR(n)$ to $\vec\Z$ implies $\lim_n\frac{r(n)}{p(n)}=1$, the above inequality implies that $\lim_n\frac{s_*(n)}{p(n)}=\lim_n\frac{e_*(n)}{p(n)}=0$ for $*\in\{l,r\}$. In particular, $\Lang$ is almost prolongable.
\end{proof}
Note that Lemma~\ref{Lemma:Technical}  holds in any, not necessarily prolongable, language.
The condition on $\frac{r(n)}{p(n)}$ is however hard to check in general.
For factorial languages, the situation is better and Theorem~\ref{Prop:Equivalence} below gives a criterion for convergence that depends only on the complexity.
For factorial and almost prolongable languages of unbounded subexponential complexity it says moreover that if the graphs converge, then they converge to the line.
We first prove that under these hypotheses one can always find a subsequence of $\Rauz(n)$ converging to $\Z$.
\begin{lemma}\label{Lem:SubExp}
Let $\Lang$ be a factorial and almost prolongable language. Suppose that its complexity is unbounded and grows subexponentially.
Then there exists a sequence of positive integers $n_d$ such that the Rauzy graphs $\Rauz(n_d)$ converge to the line $\Z$.
\end{lemma}
\begin{proof}
Let $d$ be an arbitrary positive integer.
On the one hand, we observe that for any positive subexponential function $p(n)$, there exists infinitely many $n$ such that
\begin{equation}\label{Eq:SubExp}
	\frac{p(n+2d)}{p(n)}<1+\frac1d.
\end{equation}
Indeed, otherwise all $n$ big enough would satisfy $p(n + 2d) > (1 + \frac1d)p(n)$, implying exponential growth.
On the other hand, for all $n$ big enough we have
\begin{equation}\label{Eq:SubExp2}
	\abs{\setst{v\in\Lang_n}{\exists u,w\in\Lang_d:uvw\in\Lang_{n+2d}}}>(1-\frac1d)p(n).
\end{equation}
Indeed, given $d\in\N$ one can choose $\delta>0$ such that $(1-\delta)^{2d}>(1-\frac1d)$.
Since $\Lang$ is almost prolongable, for $n$ big enough at least $(1-\delta)p(n)$ of words in $\Lang_n$ are prolongable. Using that $\Lang$ is also factorial, the left hand side in Equation~\eqref{Eq:SubExp2} is bigger than $(1-\delta)^{2d}p(n)>(1-\frac1d)p(n)$.

Consider a map $\phi\colon\Lang_{n+2d}\to\Lang_n$ that maps a word $w$ of length $n + 2d$ to the middle subword $\phi(w) = w_{d+1}\dots w_{d+n}$ of length $n$.
Since the language is factorial, the map $\phi$ is well-defined.
Assume that $n$ and $d$ are such that both Equations~\eqref{Eq:SubExp} and ~\eqref{Eq:SubExp2} are satisfied.
Then there are at most $\frac1d p(n)$ words in $\Lang_n$ which have more than one pre-image and at most $\frac1d p(n)$ words in $\Lang_n$ which have no pre-image.
Let $\Lang'_n$ denote the subset of words in $\Lang_n$ with a unique pre-image under~$\phi$.
It is then clear that $\abs{\Lang'_n} > (1-\frac2d)p(n)$.
For any word in $\Lang'_n$, its $k$-neighborhood in the Rauzy graph $\Rauz(n)$ is either a segment of $\Z$ of length $2d$ or a closed cycle of length no more than $2d$.

Now, let $(n_d)_{d\in\N}$ be an increasing sequence of integers such that for all $d$, the couple $(d,n_d)$ satisfies Equations~\eqref{Eq:SubExp} and ~\eqref{Eq:SubExp2}.
We claim that the $\Rauz(n_d)$ converge to $\Z$.
To prove that, we take any positive integer $l$ and show that with probability tending to $1$ the $l$-neighborhood of a uniformly random vertex in $\Rauz(n_d)$ is a segment of $\Z$ of length $2l$.
The number of cycles of length no more than $2l$ is finite.
More precisely, it is bounded above by some constant $c_l$ depending only on $l$ and on the cardinality of the alphabet.
Therefore, the number of all the vertices covered by these cycles is $o(p(n_d))$ due to the unboundedness of $p(n)$.
Hence, for any $d\geq l$, with probability at least $1-\frac2d - \frac{o(p(n_d))}{p(n_d)}$ the $l$-neighborhood of a vertex in $\Rauz(n_d)$ is a segment of $\Z$.
\end{proof}
Now we finally prove Theorem~\ref{Thm:Main}.
\begin{theorem}\label{Prop:Equivalence}
Let $\Lang$ be a factorial and almost prolongable language with unbounded complexity $p(n)$.
Then the following are equivalent:
\begin{enumerate}
\item The Rauzy digraphs $\orR(n)$ converge to the directed line $\vec \Z$.\label{CondF:2}
\item The Rauzy graphs $\Rauz(n)$ converge to the line~$\Z$.\label{CondF:4}
\item $p(n)$ grows subexponentially and the limit of the Rauzy digraphs $\orR(n)$ exists.\label{CondF:1}
\item $p(n)$ grows subexponentially and the limit of the Rauzy graphs $\Rauz(n)$ exists.\label{CondF:3}
\item $p(n)$ satisfies $\lim_n\frac{p(n+1)}{p(n)}=1$.\label{CondF:5}
\item The empirical spectral measures $\mu_{\Rauz(n)}$ converge to the measure that has density $\frac{1}{\pi\sqrt{4-x^2}}dx$ with respect to the Lebesgue measure on~$[-2, 2]$.\label{CondF:6}
\item $p(n)$ grows subexponentially and the empirical spectral measures $\mu_{\Rauz(n)}$ converge.\label{CondF:7}
\end{enumerate}
\end{theorem}
\begin{proof}
\textbf{\eqref{CondF:2}$\implies$\eqref{CondF:4}} and \textbf{\eqref{CondF:1}$\implies$\eqref{CondF:3}} The limit of the $\Rauz(n)$ is the underlying graph of the limit of the $\orR(n)$.

\textbf{\eqref{CondF:3}$\implies$\eqref{CondF:4}} This is Lemma~\ref{Lem:SubExp}.

\textbf{\eqref{CondF:2}$\iff$\eqref{CondF:4}$\iff$\eqref{CondF:5}}
We will in fact prove something slightly stronger. Namely, that for a general language $\Lang$ we have \eqref{CondF:2}$\iff$[\eqref{CondF:4} and $\Lang$ is almost prolongable]$\iff$[\eqref{CondF:5}, $\Lang$ is almost prolongable and $p(n)$ is unbounded].
By a theorem of Ehrenfeucht and Rozenberg~\cite{MR707633}, for a factorial language $\Lang$ either $p(n)$ is bounded or $p(n)\geq n+1$ for every $n$.
In particular, the condition $p(n)$ is unbounded is equivalent to $\lim_np(n)=\infty$.
Thus, by Lemma~\ref{Lemma:Technical}, it is enough to show that, for a factorial language, we have $\lim_n\frac{r(n)}{p(n)}=1$ if and only if $\Lang$ is almost prolongable and $\lim_n\frac{p(n+1)}{p(n)}=1$.
Observe that the last equality is equivalent to $\lim_n\frac{p(n+1)-p(n)}{p(n)}=0$.

Since $\Lang$ is factorial, every word of length $n+1$ can be obtained as $av$ with $a\in\A$ and $v\in\Lang$ of length $n$.
Hence we obtain
\[
	r_l(n)+2s_l(n)\leq p(n+1)\leq r_l(n)+ks_l(n).
\]
Using that $p(n)=r_l(n)+s_l(n)+e_l(n)$ we conclude that
\begin{equation}\label{Eq:P}
	s_l(n)\leq p(n+1)-p(n)+e_l(n)\leq (k-1)s_l(n).
\end{equation}
Analogous inequalities hold for $s_r(n)$ and $e_r(n)$.
We also have the inequality
\begin{equation}\label{Eq:I}
	p(n)-s_l(n)-s_r(n)-e_l(n)-e_r(n)\leq r(n)
\end{equation}
which holds for any language.

Suppose now that $\lim_n\frac{r(n)}{p(n)}=1$. We already know that this implies that $\Lang$ is almost prolongable. Equation~\eqref{Eq:I1} gives us $\lim_n\frac{s_l(n)}{p(n)}=\lim_n\frac{s_r(n)}{p(n)}=0$, while
by the right-hand side of Equation~\eqref{Eq:P} we obtain that $\lim_n\frac{p(n+1)-p(n)}{p(n)}=0$.

For the converse, suppose that $\Lang$ is almost prolongable and that 
$\lim_n\frac{p(n+1)}{p(n)}=1$.
The left-hand side of Equation~\eqref{Eq:P} gives us that $\lim_n\frac{s_l(n)}{p(n)}=0$ and similarly $\lim_n\frac{s_r(n)}{p(n)}=0$. Using Equation~\eqref{Eq:I} we conclude that $\lim_n\frac{r(n)}{p(n)}=1$.

\textbf{\eqref{CondF:2}$\implies$\eqref{CondF:1}}
Once again we prove something sligthly stronger than announced: for a general language $\Lang$ we have \eqref{CondF:2}$\implies$[\eqref{CondF:1}, $\Lang$ is almost prolongable and $p(n)$ is unbounded].
We have to show that Condition~\eqref{CondF:2} implies $\Lang$ is almost prolongable, and $p(n)$ is unbounded and grows subexponentially. The first part follows from~\eqref{CondF:4} and the second part from~\eqref{CondF:5}.

\textbf{\eqref{CondF:7}$\implies$\eqref{CondF:5}} Convergence of empirical spectral measures~$\mu_{\Rauz(n)}$ implies the convergence of their positive moments. The second moment of the empirical measure of a simple graph $(V, E)$ is double the average degree which is $\frac{2\abs{E}}{\abs{V}}$. In general case when the graph might have loops and double edges, the second moment of the empirical measure is 
\[
\frac{1}{\abs{V}} \sum_{v \in V} \abs{\{(e_1, e_2) \in E \colon e_1 \textnormal{ and } e_2 \textnormal{ have the same endpoints, one of which is } v\}}
\]
\[
\frac{1}{\abs{V}} \big(\sum_{v \in V} \abs{\{\textnormal{loops in } v \}}^2 +\sum_{v, u \in V, \ u\not = v} \abs{\{\textnormal{edges between } u \textnormal{ and } v\}}^2\big)
\]
In the case of Rauzy graphs, we have no more than $\abs{\A}^2$ double edges and no more than $\abs{\A}$ loops. Hence 
\begin{equation}
\int z^2 d \mu_{\Rauz(n)} = 2 \frac{p(n + 1) + O(1)}{p(n)}. \label{Eq:SecondMom}
\end{equation}
Hence, the limit $\lim_{n} \frac{p(n + 1)}{p(n)}$ exists. Due to subexponential complexity of the subshift, the only possible value of this limit is~$1$.

\textbf{\eqref{CondF:6}$\implies$\eqref{CondF:7}}
In case we know that the empirical measures $\mu_{\Rauz(n)}$ converge to the measure that has density $\frac{1}{\pi\sqrt{4-x^2}}dx$, then it follows from the Equation~\ref{Eq:SecondMom} that the complexity is subexponential.

\textbf{\eqref{CondF:4}$\implies$\eqref{CondF:6}} is done in \cite{ATV} in the general case and can easily be verified in our particular case when the limit is (a delta-measure on) a unique transitive graph,~$\mathbf{Z}$.
\end{proof}
\begin{remark}
As seen in the proof of Theorem~\ref{Prop:Equivalence}, for a general language $\Lang$, the convergence of the Rauzy digraphs $\orR(n)$ to the directed line $\vec \Z$ implies both that $\Lang$ is almost prolongable and $p(n)$ is unbounded.
Similarly, Conditions~\ref{CondF:4} and \ref{CondF:6} of Theorem~\ref{Prop:Equivalence} each implies that $p(n)$ is unbounded.
\end{remark}
Heuristically, Theorem~\ref{Prop:Equivalence} (namely, the equivalence between Conditions~\eqref{CondF:2} and~\eqref{CondF:5}) implies that for a factorial and almost prolongable language of unbounded complexity, the Rauzy digraphs converge to $\vec\Z$ (equivalently, the Rauzy graphs converge to $\Z$) if and only if the complexity is both subexponential and ``sufficiently regular''. 

As a first example of this principle, we have the following corollary of Theorem~\ref{Prop:Equivalence}, where two functions are \emph{asymptotically equivalent}, denoted $f\sim g$, if $\lim_x\frac{f(x)}{g(x)}=1$.
\begin{corollary}\label{Cor:Asympt}
Let $\Lang$ be a factorial and almost prolongable language.
Suppose that there exists a real $\alpha\geq 1$ with $p(n)\sim n^\alpha$.
Then its Rauzy digraphs converge to~$\vec \Z$.
\end{corollary}
\begin{remark}
In the above corollary, it is not enough to ask for \emph{polynomial complexity} in the usual sense that $p(n)=\Theta(n^d)$ for some integer $d$, as this only implies that there exist constants $0<C_1\leq C_2$ with $1\leq\lim_n\frac{p(n+1)}{p(n)}\leq\frac{C_2}{C_1}$.
For more on the case $p(n)=\Theta(n^\alpha)$, see Corollary~\ref{Cor:SmallCompl}.
\end{remark}

We conclude this section with the following simple observation.

\begin{proposition}
If the complexity function~$p(n)$ of a regular factorial language is subexponential then is satisfies   $\lim_{n\to\infty}\frac{p(n+1)}{p(n)} = 1.$
\end{proposition}
\begin{proof}
It is known that a regular language has either exponential or polynomial complexity (see~\cite{Stanley}). Moreover, in the latter case, the complexity function is a quasi-polynomial and, hence, satisfies $\lim_{n\to\infty}\frac{p(n+1)}{p(n)} = 1$. 
\end{proof}

\begin{corollary}
 The Rauzy graphs of any regular factorial almost prolongable language converge in the Benjamini-Schramm sense to~$\Z$.    
\end{corollary}

\subsection{Asymptotics of complexity function and convergence of Rauzy graphs}

In this subsection we present asymptotic bounds on the complexity function sufficient to conclude that the Rauzy graphs of a language converge to a line. 

Recall that $e_l(n)$ and $e_r(n)$ denote  the number of words of length $n$ that are not left, respectively not right, prolongable and $e(n)$ denotes the number of words of length $n$ that are neither left nor right-prolongable.
The estimates that we derive in this subsection are inspired by a paper by Cassaigne~\cite{MR1466182}. We need the following generalization of one of his results.
\begin{lemma}\label{Lemma:Sub}
Let $\Lang$ be a factorial and almost prolongable language with unbounded complexity and such that $e(n)=0$ for $n$ big enough.
Suppose that there exists $\alpha\in[0,\frac12]$ and $\beta\in[1,\frac32]$ such that $p(n)=O(n^\beta)$ and both $e_l(n)$ and $e_r(n)$ are $O(n^\alpha)$.
Then $p(n+1)-p(n)=O(n^{3\cdot\max(\alpha,\beta-1)})$.
\end{lemma}

Using this Lemma, we get the following corollary of our Theorem \ref{Prop:Equivalence}

\begin{corollary}\label{Cor:SmallComplGen}
Let $\Lang$ be a factorial and almost prolongable language with unbounded complexity and such that $e(n)=0$ for $n$ big enough.
Suppose that there exists $\alpha\in[0,\frac12)$, $\beta\in[1,\frac32)$ and $\gamma\in(3\cdot\max(\alpha,\beta-1),\beta]$ such that both $e_l(n)$ and $e_r(n)$ are $O(n^\alpha)$ and there exists constants $C_1$ and $C_2$ with
\begin{equation}\label{eq:betagamma}
	C_1n^\gamma\leq p(n)\leq C_2n^\beta
\end{equation}
for $n$ big enough.
Then the Rauzy digraphs of $\Lang$ converge to $\vec \Z$.
\end{corollary}
\begin{proof}
By  Lemma~\ref{Lemma:Sub}, the upper bound for $p(n)$ implies that $p(n+1)-p(n)\leq K n^{3\cdot\max(\alpha,\beta-1)}$ for some constant $K$ and for $n$ big enough.
Hence we obtain
\[
	0\leq\lim_n\frac{p(n+1)-p(n)}{p(n)}\leq \lim_n\frac{K n^{3\cdot\max(\alpha,\beta-1)}}{C_1n^\gamma}.
\]
We conclude by Theorem~\ref{Prop:Equivalence}.
\end{proof}

Observe that while $\alpha=\frac12$ and $\beta=\frac32$ are allowed in  Lemma~\ref{Lemma:Sub}, they are excluded for Corollary~\ref{Cor:SmallComplGen} as we ask there that $\gamma$ be smaller than $\beta$ and strictly bigger than $3\cdot\max(\alpha,\beta-1)$.

As a particular case of Corollary~\ref{Cor:SmallComplGen} we also have
\begin{corollary}\label{Cor:SmallCompl}
Let $\Lang$ be a factorial and almost prolongable language with unbounded complexity and such that $e(n)=0$ for $n$ big enough.
Suppose that there exists $\alpha$ and $\beta$ such that both $e_l(n)$ and $e_r(n)$ are $O(n^\alpha)$ and one of the following two conditions is satisfied:
\begin{enumerate}
\item $p(n)=O(n^\beta)$, $\alpha\in[0,\frac13)$ and $\beta\in[1,\frac43)$,
\item $p(n)=\Theta(n^\beta)$, $\alpha\in[0,\frac12)$ and $\beta\in[1,\frac32)$.
\end{enumerate}
Then the Rauzy digraphs of $\Lang$ converge to $\vec \Z$.
\end{corollary}
\begin{proof}
On the one hand, the second assertion is simply Corollary~\ref{Cor:SmallComplGen} with $\gamma=\beta$.
On the other hand, since $p(n)$ is unbounded, we have $p(n)\geq n+1$ and the first assertion is Corollary~\ref{Cor:SmallComplGen} with $\gamma=1$. But $1$ is in $(3\cdot\max(\alpha,\beta-1),\beta]$ if and only if both $\alpha<\frac13$ and $\beta<\frac43$.
\end{proof}
We end this subsection with a proof of Lemma~\ref{Lemma:Sub}.
\begin{proof}
If $\Lang$ is prolongable, then $\alpha=0$ and this is \cite[Proposition 2]{MR1466182} of Cassaigne. He also treated the case $e_l(n)\equiv 1$ and $e_r(n)\equiv 0$ by using an auxiliary language $\Lang'$ on $k+1$ letters.
We will now adapt Cassaigne's trick to the general case.
Let $z$ be a new letter not in $\A$ and let $\A'\coloneqq\A\cup\{z\}$.
Define a new language $\Lang'\subset\A'$ by
\begin{align*}
\Lang'\coloneqq\Lang&\cup\setst{z^nv}{n\in\N, v\in\Lang \textnormal{ not left-prolongable}}\\
	&\cup\setst{wz^m}{m\in\N, w\in\Lang \textnormal{ not right-prolongable}}\\
	&\cup\setst{z^muz^n}{m,n\in\N, u\in\Lang \textnormal{ neither left nor right-prolongable}}.
\end{align*}
By construction the language $\Lang'$ is factorial. We claim that it is also prolongable.
Let $t$ be in $\Lang'$. If $t$ already belongs to $\Lang$, then it is either left-prolongable in $\Lang$, or $zt$ is in $\Lang'$. In both cases, it is possible to prolong $t$ to the left in~$\Lang'$. The same holds for right-prolongation.
If $t$ is of the form $z^mv$, then $v$ is not left-prolongable in $\Lang$ and hence $z^{m+1}v$ is in $\Lang'$.
On the one hand, if $v$ is right-prolongable in $\Lang$, there exists $a\in\A$ such that $va$ is in $\Lang$. Using that $\Lang$ is factorial, we have that $va$ is not left-prolongable in $\Lang$ and hence $z^{m}va$ is in $\Lang'$.
On the other hand, if $v$ is not right-prolongable in $\Lang$ then $z^nvz$ is in $\Lang'$ by construction.
An analogous reasoning take care of the case $t=wz^n$.
Finally, words of the form $z^muz^n$ are obviously both left and right-prolongable.

We now need to compare the complexity functions of $\Lang$ and $\Lang'$.
Words of length $n$ in $\Lang$ are of the form $t$ for $t\in\A^n\cap\Lang$, $z^{n-i}v$ for $v\in\A^i\cap\Lang$ not left-prolongable, $wz^{n-j}$ for $w\in\A^j\cap\Lang$ not right-prolongable or $z^juz^{n-i-j}$ for $u\in\A^i\cap\Lang$ neither left nor right-prolongable.
In particular, any $u\in\A^i\cap\Lang$ which is neither left nor right-prolongable will add $n-i-1$ new words of the form $z^juz^{n-i-j}$ with both $j\neq 0$ and $n-i-j\neq0$.
This gives us
\[
	p_{\Lang'}(n)=p_\Lang(n)+\sum_{i=1}^{n-1}(e_{\Lang,l}(i)+e_{\Lang,r}(i)+e_\Lang(i)(n-i-1))
\]
Using that $e_\Lang(n)=0$ for $n$ big enough and that $e_{\Lang,l}(n)$ and $e_{\Lang,r}(n)$ are $O(\alpha)$ we obtain (for $n$ big enough)
\[
	p_{\Lang'}(n)\leq p_\Lang(n)+Cn^{\alpha+1}
\]
where $C$ is a constant.
Moreover, we have
\begin{align*}
	p_{\Lang'}(n+1)-p_{\Lang'}(n)&=p_{\Lang}(n+1)-p_{\Lang}(n)+e_{\Lang,l}(n)+e_{\Lang,r}(n)+\sum_{i=1}^ne_\Lang(i)\\
	&\geq p_{\Lang}(n+1)-p_{\Lang}(n).
\end{align*}
By hypothesis, there exists a real number $a>0$ and an integer $n_0$ such that we have $p_{\Lang}(n)\leq an^\beta$ for $n\geq n_0$.
Using that $\beta\geq 1$, for $n\geq n_0$ we have
\[
	p_{\Lang'}(n)\leq an^\beta+Cn^{\alpha+1}\leq (a+C)n^{\max(\alpha+1,\beta)}
\]
with $1\leq\max(\alpha+1,\beta)\leq\frac32$.
We can hence apply \cite[Proposition 2]{MR1466182} to $\Lang'$ to obtain
\[
	p_{\Lang}(n+1)-p_{\Lang}(n)\leq p_{\Lang'}(n+1)-p_{\Lang'}(n)\leq K n^{3\cdot\max(\alpha,\beta-1)}
\]
for some constant $K$ and for $n$ big enough.
\end{proof}
\subsection{Languages of one-sided subshifts}\label{Subsection:OneSided}
In this subsection, we finish the proof of Corollary \ref{Cor:SubshiftCor} and more generally show that Theorem~\ref{Thm:Main} applies to one-sided subshifts possessing a finite subset with a dense orbit.  The language of finite subwords of a one-sided subshift $\Sigma$ is always factorial and right-prolongable, but not necessarily prolongable. Hence, the only missing part is to prove that the conditions that we impose on a subshift imply almost prolongability of its language. 

There are several natural properties of a subshift that ensure the desired almost prolongability. For example, minimality or, more generally, topological transitivity is sufficient for prolongability of the language. It is known that for subshifts over a finite alphabet, topological transitivity implies existence of a dense orbit. Which, as we will see, guarantees almost prolongability. Moreover, existence of a finite set with dense orbit is enough for the same conclusion.
\begin{proposition}\label{Cor:DO}
   Let $\Sigma$ be a one-sided subshift with a finite subset with dense orbit. If $p(n)$ is unbounded then $\Lang$ is almost-prolongable. 
\end{proposition}

It is sufficient to show the following. 
\begin{lemma}\label{Lem:Prolong}
Let $\Sigma$ be a one-sided subshift with a finite set $F$ with dense orbit. Then $e_l(n)$, the number of words of length $n$ that can't be left prolongated, is bounded above by $\abs F$.
\end{lemma}
\begin{proof}
By assumption we have $\Sigma=\overline{\bigcup_{i=1}^d\setst{S^m(w_{i})}{m\in\N}}$ for some $\omega_{i}$ in $\Sigma$.
Let $v$ be a finite subword of length $n$ of some $\omega$ in $\Sigma$.
We claim that $v$ can be left-prolongated, unless maybe when $v$ is the first $n$ letters of one of the $\omega_{i}$.
If $v$ is not the first $n$ letters of a sequence $\omega$ then the claim trivially holds. It is also the case if $v$ is the first $n$ letters of $S^m(\omega_{i})$ for $m\geq 1$.
Finally, if $v$ is the first $n$ letters of some limit point $\omega$, then there exists an index $i$ and infinitely many $m$ such that $v$ is the first $n$ letters of $S^m(\omega_{i})$.
\end{proof}
The proof of Corollary~\ref{Cor:SubshiftCor} is now complete.

We conclude this subsection with a simple example of a subshift whose complexity function~$p(n)$ is unbounded and satisfies $\lim_n\frac{p(n+1)}{p(n)}=1$  but its Rauzy graphs do not converge to a line. The reason for that is the absence of almost prolongability.   
\begin{example}\label{Ex:NonAlmProl}
Let $\Lang$ be the language of a one-sided subshift $\Sigma$ over a finite alphabet~$\A$ and $z$ be a letter not in $\A$. Assume that $p_\Lang(n)$ is unbounded and $\lim_n\frac{p_\Lang(n+1)}{p_\Lang(n)}=1$ (see Section~\ref{Sec:Applications} for examples of such subshifts). 

Let $\tilde \Sigma=\Sigma\cup z\Sigma$ be a new subshift with the language $\tilde\Lang\coloneqq\Lang\cup z\Lang$. Clearly, 
the complexity function $p_{\tilde \Lang}(n)$ is $p_\Lang(n) + p_\Lang(n - 1)$ and, therefore, satisfies the limit condition. However, the language $\tilde \Lang$ is not almost prolongable and, hence, the Rauzy graphs do not converge to a line.  
\end{example}
\section{Limit of labelled Rauzy digraphs}\label{Sec:Labelled}

As we already know, if the Benjamini-Schramm limit of (unlabelled) Rauzy digraphs for a subshift of unbounded subexponential complexity exists then it is~$\vec\Z$. Throughout this section, we will assume that such a limit does exist. Now let us introduce labels on our Rauzy digraphs. Recall that we label an arc connecting two words~$u$ and~$v$ by the last letter of~$v$.  In this case, the limit of labelled Rauzy digraphs, if it exists, will be a probability measure on rooted labelled digraphs supported on $\vec\Z$ and unimodular.  The limit can then be identified with a probability measure on the set of bi-infinite words invariant under the shift, as shown in Lemma \ref{inv}. In this section, we show that if the subshift is uniquely ergodic then the Benjamini-Schramm limit of labelled Rauzy digraphs exists and is the unique ergodic invariant measure for the subshift. In the case of subshifts of linear complexity, this was earlier proved in \cite{Drummond}.

In the case of a non-uniquely ergodic subsift, one can consider subsequential limits; we show that we can always obtain an invariant measure for a subshift as such a limit with respect to some subsequence of labelled Rauzy digraphs. In Section~\ref{counterexample} we present a  concrete example of a non-uniquely ergodic system where this procedure realizes a continuum of distinct invariant measures.

In this section, we restrict our considerations to the case of topologically transitive subshifts. That is,  given an infinite word~$\omega$, we define a closed subset~$\Sigma_\omega \subset \A^\N$ to be the closure of the orbit of~$\omega$ with respect to the shift map~$S$ on $\A^\N$. We will refer to the complexity function of the language coming from $\Sigma_\omega$ as the complexity of the word $\omega$. 
Let us recall the definition of \emph{words with uniform frequencies}. 

\begin{definition}
We denote by $\abs{\omega_k\cdots\omega_{k+n}}_u$ the number of occurrences of a finite word $u$ within a finite stretch of $\omega$:
\[
	\abs{\omega_k\cdots\omega_{k+n}}_u = \sum_{i=0}^{n-\abs{u}} I_{u}(S^{i+k}\omega),
\]
where $I_u$ is the indicator function of a cylinder set with base $u$.
The~word~$\omega$ has \emph{uniform frequencies} if for every $u$ there exists a constant $F_u(\omega)$ such that 
\[
	\lim_{n}\sup_{k} \abs{\frac{\abs{\omega_k\cdots\omega_{k+n}}_u}{n+1} - F_u(\omega)} = 0.
\]
\end{definition}

\begin{proposition}[Oxtoby ergodic theorem, \cite{MR47262}]
A word $\omega$ has uniform frequencies if and only if $(\Sigma_\omega, S)$ is uniquely ergodic. 
\end{proposition}

Let us recall that the natural extension~$(\tilde \Sigma, S)$ of a subshift~$(\Sigma, S)$ is a bi-sided subshift that consists of all bi-infinite words $\tau \in \A^\Z$ such that any finite subword of~$\tau$ belongs to the language of~$\Sigma$. 
It can be shown that the natural extension of a uniquely ergodic subshift is also uniquely ergodic.
Now, let us proceed with the proof of Theorem~\ref{Thm:MainUnifFreq}.

\begin{proof}[Proof of Theorem~\ref{Thm:MainUnifFreq}]

We follow the argument from~\cite{Drummond} where the case of linear complexity was treated. First, we apply Theorem~\ref{Thm:Main} to the subshift and conclude that the Benjamini--Schramm limit of its unlabelled Rauzy digraphs is the oriented line.

Let $u$ be a finite word of even length. The Benjamini--Schramm convergence of the Rauzy digraphs implies that there exists an unbounded non-decreasing sequence $\{d_n\}$ of integer numbers such that for any sufficiently large $n$, there exists a subset $E_n$ of vertices of $\orR(n)$ such that $\abs{E_n} = o(p(n))$ and each connected component of $\orR(n) \setminus E_n$ is an oriented segment of length at least~$d_n$.

For each segment of length $k$, we assign a finite sequence of labels that can be read following the orientation starting from the vertex with no ingoing edges. This sequence forms a finite subword of $\omega$ of length $k \geq d_n$. Due to the uniform frequencies condition, the number of occurrences of~$u$ in this word is $F_u(\omega) k + o(k)$. Each such occurrence represents a vertex in $\orR(n)$ whose $\frac{\abs{u}}{2}$--neighborhood is the oriented segment whose sequence of labels is exactly~$u$. Hence, the total number of such vertices in  $\orR(n)$ is $F_u(\omega) p(n) + o(p(n))$. Therefore, the labelled Rauzy digraphs converge to a measure on $\A^\Z$ such that the cylinder set with base $u$ has measure $F_u(\omega)$. That is, this limit coincides with the unique invariant measure on $(\tilde \Sigma_\omega, S)$.

\end{proof}

We conclude this section with the following proposition showing that even if a subshift is not uniquely ergodic, we can represent some of its invariant measures as the limit of labelled Rauzy digraphs. However, to do that we have to consider only subsequential limits since the actual limit may not exist in general (see Section~\ref{counterexample}). 

\begin{proposition}\label{Prop:SubMeasure}
    Let $(\Sigma, S)$ be a subshift of unbounded subexponential complexity over a finite alphabet that has a finite subset with dense orbit. Then there exists a sequence $\{n_k\}$ such that the Benjamini--Schramm limit of the labelled Rauzy digraphs $\orR(n_k)$ exists and can be identified with an invariant probability measure on $(\tilde \Sigma,S)$. 
\end{proposition}
\begin{proof}

    First, we apply Lemma~\ref{Lem:SubExp} to $\Lang_\Sigma$ and find a sequence $\{n_k^\prime\}$ such that the underlying unlabelled Rauzy digraphs $\orR(n_k)$ converge to $\vec{\Z}$. Then, due to the compactness argument, we can choose a subsequence $\{n_k\} \subset \{n_k^\prime\}$ such that the Benjamini--Schramm limit along $n_k$ of the labelled Rauzy digraphs exists. This limit is supported on the set of labelings of  $\vec \Z$ and, by Lemma \ref{inv}, can be identified with a shift-invariant measure on $\A^\Z$. Clearly, the measure of any cylinder set with base~$u$ is positive only if $u$ belongs to the language of~$(\Sigma,S)$. Hence, the limiting measure is a shift-invariant measure supported on the natural extension of the system. 
\end{proof}

\section{Examples and applications}\label{Sec:Applications}
In this section we explain how various dynamical systems fit in the framework of our main results (Theorem ~\ref{Thm:Main}, Corollary~\ref{Cor:SubshiftCor} and Theorem \ref{Thm:MainUnifFreq}).

\subsection{Symbolic dynamical systems}

Recall that a \emph{substitution} is a mapping from an alphabet $A$ to the set of words $A^{*}$. It can be extended to a monoid morphism $A^{*}\to A^{*}$ and further to a morphism $\phi: A^\mathbf N \to A^\mathbf N$.  A \emph{pure morphic sequence} is a fixed point of such a morphism. A natural generalization of this notion is that  of a \emph{morphic sequence}, that is, the image of a pure morphic sequence under some coding (a morphism $A\rightarrow B$ between two alphabets).

It was proved in~\cite{Dev} that the complexity function~$p$ of a morphic sequence is either~$p(n)= \Theta(n^{1+1/k})$ for some $k\in\N$, or $p(n) = O(n\log n)$.  
Rauzy graphs in the linear complexity case (corresponding to primitive substitutions) were studied in the literature (see, for example, \cite{Frid}). We are now able to identify the limit of the Rauzy graphs for morphic sequences of superlinear complexity. Indeed, one can easily check that if $k>2$, all the conditions of Theorem \ref{Thm:Main} and Corollary \ref{Cor:SubshiftCor} are satisfied. 
Therefore, we prove the following 
\begin{corollary}
    For $k > 2$, the Benjamini-Schramm limit of the Rauzy graphs corresponding to a morphic sequence~is~$\Z$.
\end{corollary}
Note, however, that there are examples (see \cite{Dev})  with $k=1$, and so not all the morphic sequences are covered by the Corollary.

The study of morphic sequences and their invariant and/or ergodic measures is an important step in the understanding of the dynamical properties of tiling systems and of quasi-crystals (see, for example, \cite{Be}). In certain special cases the tiling space of a word has been shown to be homeomorphic to the inverse limit of its Rauzy graphs (\cite{Ju}); the Benjamini-Schramm limit of the Rauzy graphs provides therefore an invariant measure on the associated tiling space.

Next, we consider Toeplitz subshifts. 
Theorem A in \cite{Kos98} claims that for any rational number $\frac{p}{q}$ and any function $f(n)=o(n^\alpha)$ such that $nf'(n)=o(n^\alpha)$ there exists a Toeplitz sequence whose complexity satisfies the following condition:
\[c_{1}f(n)n^{\frac{p}{q}}\le p(n)<c_{2}f(n)n^{\frac{p}{q}}.\]
So, if we takes parameters $\alpha$ and $\frac{p}{q}$ such that
$\alpha + \frac{p}{q} < \frac43$ we can apply Corollary~\ref{Cor:SubshiftCor} and conclude that the Benjamini-Schramm limit of the Rauzy graphs of the subshift is the line. It also holds for the so-called simple Toeplitz subshifts considered in \cite{Se18} as they have linear complexity. 

Another family of so called $(p,q)$-Toeplitz shifts is given in \cite{CK97}. It is shown that this family has polynomial complexity (see Theorem 7 in \cite{CK97}) and it follows easily from the proof of Lemma 8 in \cite{CK97} that the asymptotics of $p(n+1)-p(n)$ is given by $O(n)$. Thus, for this series of Toeplitz subshifts, the Benjamini-Schramm limit of corresponding Rauzy graphs is also~$\Z$.

\subsection {Interval Translation Mappings}
Interval translation mappings (ITMs) were introduced by M. Boshernitzan and I. Kornfeld (see \cite{BoKo}) as a natural generalization of one of the most famous one-dimensional dynamical systems - interval exchange trans\-for\-mations (IETs). IETs were defined in 60s in connection with the study of measured foliations on oriented surfaces and billiards in rational polygons. Namely, IETs appear as the first return map to the transversal for the leaves of the measured foliation (or, equivalently, for the vertical translation flow on the translation surface or billiard flow in a rational billiard). We refer to \cite[Chapter 5]{CFS} for basics about IETs.

\begin{definition}
    \emph{Interval exchange transformation} (IET) is a bijective piecewise continuous map from the half-interval of real axis to itself with a finite number of discontinuity points, such that on each interval of continuity this map is a translation. 
    
	An \emph{interval translation map} is a piecewise translation map $T$ defined on an half-open interval $I \subset \field R$ with values in $I$.
	We call $T$ a $d$-interval translation map (or $d$-ITM) if $I$ has exactly $d$ maximal half-open sub-intervals to which the restriction of the $T$ is a translation.
\end{definition}

As follows from the definition, both IETs and ITMs are defined by a finite number of parameters. 

From the point of view of symbolic dynamics, IETs admit the following natural description: each interval of continuity can be denote by a separate letter, and thus any orbit is given by the (possibly, periodic) word of a finite alphabet, and the application of the IET is given by a one-side shift on such a word. It is easy to estimate and very often even to compute complexity of an IETs on $d$ intervals: $p(n)\le (d-1)n+1$ and it is a precise equality for all generic maps (generic means that the lengths of the intervals of continuity are rationally independent).  Such IETs are known to be minimal and uniquely ergodic. Therefore Drummond's result \cite{Drummond} applies and gives us a description of the unique invariant ergodic measure as the limit of (labelled) Rauzy digraphs.

Boshernitzan and Kornfeld (\cite{BoKo}) noticed that ITMs can be divided into two classes in the following way: 
looking at the sequence $\Sigma_k = I\cap T(I)\cdots \cap T^{k}(I)$
we distinguish two cases. Either there exists $m\in\mathbf{N}$ such that $\Sigma_m = \Sigma_{m+1}$, then we say that an ITM $T$ is \emph{of finite type}. Morally it means that $T$ can be reduced to an IET. If such an $m$ does not exist and the limit of $\Sigma_m$ is a nowhere dense set, we say that $T$ is \emph{of infinite type}. 
A challenging question is to understand the  maps of infinite type because only these ITMs represent some dynamics crucially different from IETs. 
Additional motivation to study ITMs of infinite type comes from the geometric group theory, where a very similar object, named \emph{systems of isometries} or \emph{band complex}, appears in connection with actions of free groups on $\mathbf{R}$-trees (see \cite{GLP} for the definitions); ITMs of infinite type can be seen as a particular case of systems of isometries of so called thin type. 

ITMs admit symbolic description  similar to IETs:  encode each interval of continuity by a letter, and hence get every orbit encoded by an infinite word in this finite alphabet and the ITM acting by the one-sided shift. 

Bruin and Troubetzkoy (\cite{BTr}) studied a 2-parameter family of ITM's of linear complexity, where a typical map of infinite type is uniquely ergodic (\cite{Skr}). Hence Drummond's result \cite{Drummond} applies and gives a description of the unique ergodic invariant measure as the limit of the labelled Rauzy digrpahs.

In general, the complexity of an ITM is at most polynomial \cite{BoKo}. The next candidate for a similar analysis is the class of the so-called ''double rotations"  since we know that typical double rotation of infinite type is uniquely ergodic (\cite{ArFoHuSkr}). However, no exact estimate on the complexity of double rotations of infinite type is known.

\subsection{Polygonal and polyhedral billiards}
An important class of dynamical systems that fit in the framework of Theorem ~\ref{Thm:Main} but is not covered by Drummond's results are billiards in rational polygons and, more generally, polyhedra. 

Given a polygon $Q$ with a finite number of sides, we can consider a billiard ball, i.e. a point mass, that moves inside the polygon $Q$ with unit speed along a straight line until it reaches the boundary, then instantaneously changes direction according to the mirror law. Billiards are widely studied in different branches of mathematics, for example, in connection with Hamiltonian systems or geodesic flows on Riemannian manifolds (see, for example, \cite{Ta}). 
We call the polygon $Q$ \emph{rational}, if all the angles of $W$ are equal to $\frac{p}{q}\pi$ for some natural $p$ and $q$. Billiards in rational polygons play a particular role in Teichm\"uller dynamics due to their connections with interval exchange transformations and measured foliations on surfaces.

A billiard dynamical system admits a natural coding: we label the sides of the polygon by symbols from a finite alphabet whose cardinality is equal to the number of sides of $Q$ and code the orbit by the sequence of sides it hits. Application of the billiard map in a given direction defines a shift of the word provided by the orbit.
Then there are two natural ways to define a subshift via this coding. The \emph{general subshift}~$\Sigma^g$ consists of the infinite words corresponding to all the orbits of the billiard and the \emph{directional subshift}~$\Sigma^{d,v}$ consists of words corresponding to all the orbits in a given direction~$v$. 
These subshifts give rise to two different approaches toward the study of complexity of billiard systems. The \emph{general complexity}~$p_g(n)$ (i.\,e., complexity of $\Sigma^g$) includes all the directions, while the \emph{directional complexity}~$p_{d,v}(n)$ (i.\,e., complexity of $\Sigma^{d,v}$) counts only trajectories in one (typically, irrational) direction~$v$. 

It is known that directional complexity for the rational polygonal billiard is linear (\cite{Hu95}), and, hence, we can apply~\cite{Drummond}. 
The general complexity of a rational polygonal billiard grows polynomially.
It was proved in~\cite{CaHuTr02} that for any rational convex polygon, there exist positive constants $D_1$, $D_2$ such that
\begin{equation}\label{eq_polygons}
    D_1<\frac{p_g(n)}{n^3}<D_2.
\end{equation}
Moreover, for certain billiard tables such as the unit square, the isosceles right triangle, or the equilateral triangle the limit $\lim_{n\to\infty}\frac{p_g(n)}{n^3}$ exists (see \cite{CaHuTr02}). Let us note that there are no known examples when this limit does not exist.
Using relation~\eqref{eq_polygons} together with the bounds by Masur~\cite{Ma} we obtain that for any rational convex polygon $Q$ 
\[
	\lim_{n\to\infty}\frac{p_g(n+1)}{p_g(n)}=1.
\]
Since the language of $\Sigma^g$ is prolongable we obtain the following. 
\begin{corollary}
    The Benjamini-Schramm limit of the Rauzy graphs, associated with the language of all trajectories of the billiard flow in the rational polygon, is equal to $\Z$.
\end{corollary}

Not much is known about the complexity of polyhedral billiards. However, we can apply our results to directional subshifts corresponding to hypercubic billiards.  
Let $Q$ be a hypercube of dimension~$s+1$ and let~$v$ be a totally irrational direction (i.\,e., such
that its components are independent over $\Q$ as well as their inverses). It is known that the billiard flow in~$\Q$ in direction~$v$ is uniquely ergodic as well as the corresponding directional subshift (see, e.\,g.,~\cite{ArMaShiTa94} for discussion of the three-dimensional case). Moreover, the following explicit polynomial formula for the directional complexity was shown in~\cite{Ba}
\[
	p_{d, v}(n) = \sum_{i=0}^{inf(n,s)} \frac{n!s!}{(n-i)!i!(s-i)!}.
\]
Thus, $p_{d, v}(n) \sim n^{s}$ and billiard trajectories for a given irrational direction are also covered by our main results. Applying Corollary~\ref{Cor:SubshiftCor} and Theorem~\ref{Thm:MainUnifFreq}, we get the following 
\begin{proposition}
The Rauzy graphs of the language of hypercubic billiard trajectories in a given totally irrational direction converge to~$\Z$. Moreover, for almost every such direction, the unique invariant measure for the directional subshift  can be identified with the limit of the labelled Rauzy digraphs. 
\end{proposition}

\section{A non uniquely ergodic example and subsequential limits of Rauzy digraphs}\label{counterexample}
We conclude the paper with an example of a non-uniquely ergodic system where we identify a continuum of invariant measures with subsequential limits of labelled Rauzy digraphs. 
In a recent paper~\cite{CK22}, Cassaigne and Kabor\'e constructed a uniformly recurrent infinite word $u$ with complexity bounded by $3n+1$ that does not have uniform frequencies; see also \cite{Keane} for an older example with similar properties. The linear bound on the complexity implies that the unlabelled Rauzy digraphs converge to $\vec\Z$ due to Corollary~\ref{Cor:SubshiftCor}.   

Their construction of such a word goes as follows.
For certain very fast growing sequences $(l_i)_i$, $(m_i)_i$ and $(n_i)_i$, consider two sequences $(u_i)_i$ and $(v_i)_i$ of finite binary words defined by the rule $u_0 = 0$, $v_0 = 1$ and for any non-negative integer $i$, $u_{i+1} = u_i^{m_i}v_i^{l_i}$ and  $v_{i+1} = u_i^{m_i}v_i^{n_i}$.
The word $u$ is the limit of finite words $u_i$'s.
The exact values of the parameters that can be used for our purpose are given by
\[
	l_i = 2^{2\cdot2^i+4}, \quad m_i = 2^{8\cdot2^i}\quad \textnormal{and}\quad n_i = 2^{10\cdot2^i}.
\]
 It is proved in~\cite{CK22} that the complexity $p(n)$ of the word $u$ is bounded from above by $3n+1$ and, moreover, $p(n+1) - p(n) \leq 3$.
 At the same time, they show that $\liminf \frac{p(n)}{n} = 2$.
 This equality implies that the corresponding subshift has at most $2$ invariant ergodic probability measures \cite{Boshernitzan1984}.
 It is also shown in~\cite{CK22} that $u$ does not have uniform frequencies and, therefore, the corresponding subshift is not uniquely ergodic and has hence exactly two invariant ergodic probability measures. 
 The key observation is that $u_i$'s and $v_i$'s have different frequencies of $0$:
\[
 	\lim\limits_{i \to \infty} \frac{\abs{u_i}_0}{\abs{u_i}} = \alpha > \lim\limits_{i \to \infty} \frac{\abs{v_i}_0}{\abs{v_i}} = \beta.
\]
One can see two ergodic invariant probability measures by considering two weak limits:
\[
	\mu_\alpha = \lim \frac{1}{\abs{u_i}} \sum_{k = 0}^{\abs{u_i}-1} \delta\{S^k u\}
	\quad\textnormal{and}\quad 
	\mu_\beta = \lim \frac{1}{\abs{v_i}} \sum_{k = 0}^{\abs{v_i}-1} \delta\{S^k (S^{m_{i+1}\abs{u_i}} u)\}.
\]
Measures $\mu_\alpha$ and $\mu_\beta$ have different values ($\alpha$ and $\beta$ respectively) on the cylinder set $\{0\} \times \{0,1\}^\N$. Hence every invariant measure is determined by its value on this cylinder set.

Let $\Lang_n^0$ be the set of subwords of $u$ of length $n$ ending with $0$.  Assuming the limit of the labelled Rauzy digraphs exists, we have the convergence of $\frac{\abs{\Lang_n^0}}{\abs{\Lang_n}}$ to the measure of those digraphs whose unique outgoing edge from the origin is labelled with~$0$.

 The complexity $p(n)$ exhibits different regimes of behavior in which different values of  $\frac{\abs{\Lang_n^0}}{\abs{\Lang_n}}$ may occur.
 First, assume that $n$ is such that  $a_i \coloneqq 2l_i \abs{v_i} < n < \tfrac{m_i \abs{u_i}}{2} \eqqcolon b_i$. Let $w_i$ be the greatest common suffix of $u_i$ and $v_i$ and let $\prefix(x, k)$ denote the prefix of length~$k$ of the word~$x$.
 Let us call a word $x$ \emph{left special} if it is possible to prolong it to the left in at least two different ways. 
The following words are left special and distinct: $\prefix(w_i v_i^{l_i} u_i^{m_i}, n)$, $\prefix(w_i v_i^{n_i -1 }, n)$ and $\prefix(w_i u_i^{m_i -1}, n)$.
We claim that there are no more than three left special words of length $n$, and hence that the above words are all of them.
Indeed, since the language of subwords of~$u$ is prolongable, the difference $p(n + 1) - p(n)\leq 3$ is equal to the number of left special words in $\Lang_n$.
Similarly, the difference $\abs{\Lang_{n+1}^0} - \abs{\Lang_n^0}$ is equal to the number of left special words in $\Lang_n$ that end with~$0$. Applying this to every $n \in [a_i, b_i]$, we obtain 
 \[
 	p(b_i) = p(a_i) + 3 (b_i - a_i) = 3 b_i(1 + o(1)),
\]
where the second equality follows from the specific choice of $l_i$, $m_i$ and $n_i$ and the fact that $\tfrac{\abs{v_i}}{\abs{u_i}}\leq\tfrac{n_{i-1}}{l_{i.1}}$.
We also have
\[
	 \abs{\Lang_{b_i}^0} = \abs{\Lang_{a_i}^0} + \abs{A}_0 + \abs{B}_0 + \abs{C}_0,
\]
 where $A$, $B$, and $C$ are certain $(b_i - a_i)$-subwords of $u_i^{m_i}$, $v_i^{n_i -1 }$ and $u_i^{m_i -1}$ respectively and $\abs{A}_0$ is the number of occurence of the letter $0$ in the word $A$. Therefore,
\[
	 \abs{\Lang_{b_i}^0} = \abs{\Lang_{a_i}^0} + (2\alpha + \beta)(b_i - a_i) (1 + o(1)) = (2\alpha + \beta)b_i(1 + o(1)).
 \]
Finally for this regime we have $\frac{\abs{\Lang_{n}^0}}{\abs{\Lang_{n}}} = (\frac{2}{3}\alpha +  \frac{1}{3}\beta) (1 + o(1))$.

Similar arguments applied to the case $2 m_i\abs{u_i} < n < \tfrac{n_i \abs{v_i}}{2}$ produce the following three left special words: $\prefix(w_i v_i^{l_i}  u_{i+1}^{m_{i+1} - 1}, n)$, $\prefix(w_i v_i^{l_i} u_i^{m_i}v_i^{n_i}, n)$ and $\prefix(w_i v_i^{n_i -1 }, n)$. The same counting argument shows that for this regime we have
\[
	\frac{\abs{\Lang_{n}^0}}{\abs{\Lang_{n}}} = \Bigl(\frac{1}{3}\alpha +  \frac{2}{3}\beta\Bigr) (1 + o(1)).
\]
Therefore, the limit of the labelled Rauzy digraphs can not exist for this symbolic system.

We can consider subsequential limits instead. Such limits can also be identified with invariant measures on the natural extension provided the limit of the unlabelled Rauzy digraphs is $\vec\Z$. Subsequential limits always exist due to the weak compactness of the space of probability measures on a compact metric space of rooted digraphs. Moreover, in this case, we have a continuum of such limits.
In our case, we have only two ergodic invariant  measures $\mu_\alpha$ and $\mu_\beta$, hence every invariant measure is a convex combination $\theta \mu_\alpha + (1 - \theta) \mu_\beta$. We will denote such measure by $\mu_\gamma$, where $\gamma = \theta \alpha + (1 - \theta) \beta$.
\begin{proposition}\label{Prop:CKexample}
    For any $\gamma \in [\frac{1}{3}\alpha +  \frac{2}{3}\beta, \frac{2}{3}\alpha +  \frac{1}{3}\beta]$ the measure $\mu_{\gamma}$ can be obtained as a subsequential limit of labelled Rauzy digraphs.
\end{proposition}
\begin{proof}
    The existence of the subsequential limits, when $\gamma$ is an endpoint of the interval, follows from the arguments above.
For any $\gamma$ in the interior and for any $N$, we can find some $n > N$ such that
$\bigl|\frac{\abs{\Lang_{n}^0}}{\abs{\Lang_{n}}} - \gamma\bigr| < \frac1N$.
Indeed, the value of $\frac{\abs{\Lang_{n}^0}}{\abs{\Lang_{n}}}$ is oscillating between $\frac{1}{3}\alpha +  \frac{2}{3}\beta$ and $\frac{2}{3}\alpha +  \frac{1}{3}\beta$, and the difference 
$\frac{\abs{\Lang_{n+1}^0}}{\abs{\Lang_{n+1}}} - \frac{\abs{\Lang_{n}^0}}{\abs{\Lang_{n}}}$ is bounded by $\frac{3}{\abs{\Lang_{n}}}$ which tends to zero. Taking a subsequential limit over such $n$'s we obtain the measure $\mu_\gamma$.
\end{proof}

There are several natural questions that arise in view of Propositions~\ref{Prop:CKexample} and~\ref{Prop:SubMeasure}. In the above Cassaigne-Kabor\'e example, can one obtain ergodic measures  $\mu_{\alpha}$ and $\mu_{\beta}$ as subsequential Benjamini--Schramm limits of labelled Rauzy digraphs? More generally, can one get an ergodic measure as such subsequential limit for any minimal subshift with subexponential complexity? Can one characterize geometric properties of the set of all such invariant measures?
%
%
%
%
%
%
%
%
%

\vspace{4ex}
\begin{center}
\textsc{Paul-Henry Leemann, Department of Pure Mathematics, Xi'an Jiaotong--Liverpool University, Suzhou, P.R. China}\\
\textit{E-mail address: }\texttt{PaulHenry.Leemann@xjtlu.edu.cn}\\[2ex]

\textsc{Tatiana Nagnibeda, Section de math\'ematiques, Universit\'e de Gen\`eve, 1205~Gen\`eve, Switzerland}\\
\textit{E-mail address: }\texttt{Tatiana.Nagnibeda@unige.ch}\\[2ex]

\textsc{Alexandra Skripchenko, Faculty of Mathematics, National Research University Higher School of Economics, Usacheva St. 6, 119048 Moscow, Russia \textit{and}
Skolkovo Institute for Science and Technology, Skolkovo Innovation Center, 143026 Moscow, Russia}\\
\textit{E-mail address: }\texttt{sashaskrip@gmail.com}\\[2ex]

\textsc{Georgii Veprev, Section de math\'ematiques, Universit\'e de Gen\`eve, 1205~Gen\`eve, Switzerland}\\
\textit{E-mail address: }\texttt{georgii.veprev@gmail.com}\\[2ex]
\end{center}
\end{document}